\documentclass[11pt]{amsart}

\usepackage{latexsym,enumerate}
\usepackage{amssymb, xcolor,mathrsfs}
\usepackage{amsmath,amsthm,amsfonts,amssymb,latexsym, mathabx}
\usepackage[colorlinks,pagebackref]{hyperref}
\usepackage[utf8]{inputenc}

\newtheorem{theorem}{Theorem}[section]
\newtheorem{lemma}[theorem]{Lemma}
\newtheorem{proposition}[theorem]{Proposition}

\newtheorem{conjecture}[theorem]{Conjecture}

\theoremstyle{definition}
\newtheorem{definition}[theorem]{Definition}

\theoremstyle{remark}
\newtheorem{remark}[theorem]{Remark}

\numberwithin{equation}{section}

\newcommand{\Ker}{\operatorname{Ker}}
\newcommand{\Aut}{{\mathrm {Aut}}}

\newcommand{\Irr}{{\mathrm {Irr}}}

\newcommand{\ord}{{\mathrm {ord}}}

\newcommand{\diag}{{\mathrm {diag}}}

\newcommand{\Syl}{{\mathrm {Syl}}}

\newcommand{\lcm}{{\mathrm {lcm}}}
\newcommand{\Gal}{{\it Gal}}

\newcommand{\CC}{{\mathbb C}}

\newcommand{\QQ}{{\mathbb Q}}
\newcommand{\ZZ}{{\mathbb Z}}

\newcommand{\FF}{{\mathbb F}}

\newcommand{\EC}{\mathcal{E}}

\newcommand{\lev}{\mathrm{\mathbf{lev}}}

\newcommand{\bC}{{\mathbf{C}}}
\newcommand{\bG}{{\mathbf{G}}}

\newcommand{\bT}{{\mathbf{T}}}
\newcommand{\bL}{{\mathbf{L}}}
\newcommand{\bO}{{\mathbf{O}}}
\newcommand{\bN}{{\mathbf{N}}}
\newcommand{\bZ}{{\mathbf{Z}}}

\def\nor{\trianglelefteq\,}

\begin{document}

\title[The $p$-rationality level and a lower bound for $p'$-degree characters]
{The continuity of $p$-rationality and a lower bound for $p'$-degree
characters of finite groups}

\author[N.\,N. Hung]{Nguyen Ngoc Hung}
\address{Department of Mathematics, The University of Akron, Akron,
OH 44325, USA} \email{hungnguyen@uakron.edu}

\thanks{We thank Attila Mar\'{o}ti, Gabriel Navarro, and Jay Taylor
for useful conversations on topics related to this work. Special
thanks to Gunter Malle for his many helpful comments and corrections
that have greatly improved the exposition in the paper. Finally, we
thank the referee for careful reading and helpful comments on the
paper.}



\begin{abstract} Let $p$ be a prime and $G$ a finite group. We
propose a strong bound for the number of $p'$-degree irreducible
characters of $G$ in terms of the commutator factor group of a Sylow
$p$-subgroup of $G$. The bound arises from a recent conjecture of
Navarro and Tiep \cite{Navarro-Tiep21} on fields of character values
and a phenomenon called the continuity of $p$-rationality level of
$p'$-degree characters. This continuity property in turn is
predicted by the celebrated McKay-Navarro conjecture
\cite{Navarro04}. We achieve both the bound and the continuity
property for $p=2$.
\end{abstract}

\maketitle


\section{Introduction}

For a character $\chi$ of a finite group $G$, let $\QQ(\chi)$ denote
the smallest field containing all the values of $\chi$, which is
often referred to as the \emph{field of values} of $\chi$. The
smallest positive integer $f$ such that
$\QQ(\chi)\subseteq\QQ_f:=\QQ(e^{2\pi i/f})$ is called the
\emph{conductor} of $\chi$ and denoted by $c(\chi)$. (This $c(\chi)$
used to be called the Feit number of $\chi$ \cite[p. 52]{Navarro18},
but the name conductor appears more in recent literature.) To
measure how $p$-rational a character $\chi$ is, one considers
$c(\chi)_p$ -- the largest $p$-power divisor of $c(\chi)$. The
nonnegative integer
\[\lev(\chi)=\lev_p(\chi):=\log_p(c(\chi)_p)\]
is called the $p$-\emph{rationality level} of $\chi$. In a certain
sense, a character is more $p$-irrational if it has higher
$p$-rationality level. For instance, the usual $p$-rational
characters (\cite[Definition 6.29]{Isaacs1}) are precisely
characters of level $0$, and the so-called almost $p$-rational
characters are the ones of level at most $1$ (see
\cite{Isaacs-navarro-Sangroniz,Hung-Malle-Maroti}).

The $p$-rationality level arises naturally in the context of the
celebrated McKay-Navarro (MN) conjecture \cite{Navarro04} (see
Section \ref{sec:MNconj} for more details). Let $\Irr_{p'}(G)$ and
$\Irr_{p'}(\bN_G(P))$ respectively denote the set of $p'$-degree
irreducible characters of $G$ and of the normalizer $\bN_G(P)$ of a
Sylow $p$-subgroup $P$ of $G$. The MN conjecture implies that the
number of characters in $\Irr_{p'}(G)$ and $\Irr_{p'}(\bN_G(P))$ at
each $p$-rationality level is the same, therefore refining the
well-known McKay conjecture \cite{McKay}, which asserts that
$|\Irr_{p'}(G)|=|\Irr_{p'}(\bN_G(P))|$.

This paper has two objectives that do not appear at first sight to
be connected. The first is to propose a $p$-local lower bound for
the number $|\Irr_{p'}(G)|$ and the second is to present a
phenomenon that we call the \emph{continuity of $p$-rationality} of
irreducible characters of $p'$-degree. The second is predicted by
the MN conjecture (see Theorem \ref{thm:McKay-Navarro}) and the
first is implicitly suggested by the second (see Theorem
\ref{thm:bound}). At the end, we prove both for the prime $p=2$.

\begin{conjecture}\label{conj:continuity1}
Let $p$ be a prime, $G$ a finite group and $\alpha\in \ZZ^{\geq 2}$.
If $G$ has an irreducible $p'$-degree character of $p$-rationality
level $\alpha$, then $G$ has irreducible $p'$-degree characters of
every level from $2$ to $\alpha$.
\end{conjecture}

\begin{theorem}\label{thm:main1}
Conjecture \ref{conj:continuity1} holds true for $p=2$.
\end{theorem}

The works of G. Navarro and P.\,H. Tiep \cite{Navarro-Tiep19} and of
G. Malle \cite{Malle19} already show that, when the commutator
factor group of a Sylow $2$-subgroup $P$ of $G$ has exponent at
least 4, the largest $2$-rationality level of an odd-degree
irreducible character of $G$ is precisely $\log_2(\exp(P/P'))$ (see
Theorem \ref{thm:Malle-Navarro-Tiep19}). Therefore, Theorem
\ref{thm:main1} can be reformulated as:

\begin{theorem}\label{thm:continuity}
Let $p=2$, $G$ a finite group and $P\in\Syl_p(G)$. Suppose that
$\log_p(\exp(P/P'))=\alpha\geq 2$. Then $G$ has irreducible
$p'$-degree characters of every $p$-rationality level from $2$ to
$\alpha$ (and none of level higher than $\alpha$).
\end{theorem}

Conjecture \ref{conj:continuity1} together with the recent works
\cite{Navarro-Tiep21} on fields of values of $p'$-degree characters
and \cite{Hung-Malle-Maroti} on bounding the number of almost
$p$-rational $p'$-degree characters imply the following lower bound
for the number of $p'$-degree irreducible characters in a finite
group.

\begin{conjecture}\label{conj:Irrp'-bound}
Let $p$ be a prime, $G$ a finite group of order divisible by $p$,
and $P\in\Syl_p(G)$. Then
\[|\Irr_{p'}(G)|\geq \frac{\exp(P/P')-1}{p-1}+2\sqrt{p-1}-1.\]
\end{conjecture}

Thanks to Theorem \ref{thm:continuity}, we have the following.

\begin{theorem}\label{thm:Irr2'-bound}
Let $G$ be finite group and $P\in\Syl_2(G)$. Then
\[|\Irr_{2'}(G)|\geq \exp(P/P').\] Moreover, the equality occurs if
and only if $P$ is cyclic and self-normalizing.
\end{theorem}

\begin{remark}
In general the continuity property fails at level $1$. On one hand,
as $\QQ_n=\QQ_{2n}$ for any odd positive integer $n$, there is no
character with $2$-rationality level $1$. On the other hand, for odd
$p$, the simple group $PSL_2(8)$, for instance, has three
irreducible characters of $3$-rationality level $2$ but none of
level $1$. The existence of level $1$ when $p>2$ depends on the
action of $\bN_G(P)/P$ on $P/P'$ and the precise connection is not
known to us at the time of this writing. Other than this exception,
Conjecture \ref{conj:continuity1} and Isaacs-Navarro's Conjecture
\ref{conj:Isaacs-Navarro} would provide complete control of the
$p$-rationality level of $p'$-degree characters of a finite group,
solely in terms of a Sylow $p$-subgroup.
\end{remark}

\begin{remark} The bound $|\Irr_{2'}(G)|\geq \exp(P/P')$ can be refined by partitioning
the set $\Irr_{2'}(G)$ according to the $2$-rationality level, and
this indeed is how we came up with the bound in Conjecture
\ref{conj:Irrp'-bound}. When $\exp(P/P')=2^\alpha\geq p^2$, as shown
in Section \ref{sec:lower-bound}, $G$ will have at least two
$2$-rational odd-degree irreducible characters and, for each $2\leq
\beta\leq \alpha$, at least $2^{\beta-1}$ odd-degree irreducible
characters of $2$-rationality level $\beta$.
\end{remark}

\begin{remark}
Given that the McKay conjecture is now known for $p=2$ (thanks to
the work of G. Malle and B. Sp\"{a}th \cite{MS16}), Theorem
\ref{thm:Irr2'-bound} of course can also be argued by assuming that
$P\nor G$, and so it is reduced to showing that the conjugacy class
number $k(AM)$ of the semidirect product of an odd-order group $M$
acting on an abelian $2$-group $A$ is at least $\exp(A)$. However,
even in this much simpler situation, we are not aware of any proof
that does not use the idea of $2$-rationality level outlined in the
previous remark.
\end{remark}

\begin{remark}\label{remark:1}
The bound proposed in Conjecture \ref{conj:Irrp'-bound}
substantially improves the known bound $|\Irr_{p'}(G)|\geq
2\sqrt{p-1}$ established by G. Malle and A. Mar\'{o}ti in
\cite{Malle-Maroti}. In fact, Malle-Mar\'{o}ti's bound was improved
in \cite{Hung-Malle-Maroti} where the same bound is obtained but for
the number of almost $p$-rational $p'$-degree characters instead.
Taking this into account, Conjecture \ref{conj:Irrp'-bound} is
reduced to showing that, when $|G|$ is divisible by $p$, the number
of $p'$-degree irreducible characters of $G$ of level 2 or higher is
at least $\frac{\exp(P/P')-1}{p-1}-1$.
\end{remark}

\begin{remark} We reduce Conjecture \ref{conj:continuity1} to a
problem on almost simple groups (Theorem \ref{thm:reduction}), and
then make use of the classification of finite simple groups to
achieve the results for $p=2$. Along the way, it is shown that the
$p$-rationality level of a $p'$-degree irreducible character $\chi$
of a finite group of Lie type in characteristic not equal to $p$ is
determined solely by the $p$-part of a semisimple element $s$ (in
the dual group) defining the Lusztig series containing $\chi$, under
a certain condition on the centralizer of $s$, see Theorem
\ref{thm:technical}. This may be of independent interest and will be
useful in future study of $p$-rationality of character values. In
fact, it has been used in \cite{Hung22} to confirm a conjecture on
values of Sylow restrictions of $p'$-degree characters
(\cite[Conjecture C]{Navarro-Tiep21}) for certain classical groups.
\end{remark}

The layout of the paper is as follows. In Section \ref{sec:MNconj},
we explain how Conjecture \ref{conj:continuity1} follows from the MN
conjecture and mention some other consequences related to
$p$-rationality level. In Section \ref{sec:3} we introduce the
notion of $p$-rationality level of a character value and prove a key
result on the level of a character product. Section
\ref{sec:odd-degree} is devoted to the proofs of Theorems
\ref{thm:main1} and \ref{thm:continuity}, assuming a continuity
result for almost simple groups that is proved in Sections
\ref{sec:simple-groups}, \ref{sec:linear-unitary} and \ref{sec:E6}.
Finally, we discuss Conjecture \ref{conj:Irrp'-bound} and prove
Theorem \ref{thm:Irr2'-bound} in Section~\ref{sec:lower-bound}.


\section{The McKay-Navarro conjecture and $p$-rationality}\label{sec:MNconj}

In this section we explain that the continuity of $p$-rationality of
$p'$-degree characters indeed follows from the McKay-Navarro
conjecture.

Our notation is fairly standard, following \cite{Isaacs1,Navarro18}
on character theory and \cite{Carter85,Digne-Michel91} on
representation theory of finite groups of Lie type.

Let $p$ be a prime and $G$ a finite group. As usual let $\Irr(G)$
denote the set of all irreducible ordinary characters of $G$, and
recall that
\[\Irr_{p'}(G):=\{\chi\in\Irr(G):p\nmid \chi(1)\}.\] The well-known
McKay conjecture \cite{McKay} asserts that the number of irreducible
$p'$-degree characters of $G$ equals to that of the $p$-local
subgroup $\bN_G(P)$, which is the normalizer of some
$P\in\Syl_p(G)$; that is,
\[
|\Irr_{p'}(G)|=|\Irr_{p'}(\bN_G(P))|.
\]
In \cite{Navarro04}, G. Navarro brought Galois automorphisms into
the McKay conjecture by proposing that there exists a bijection from
$\Irr_{p'}(G)$ to $\Irr_{p'}(\bN_G(P))$ that commutes with the
action of a certain subgroup $\mathcal{H}$ (depending on $p$) of the
Galois group $\Gal(\QQ_{|G|}/\QQ)$. The subgroup $\mathcal{H}$
consists of automorphisms that send every root of unity
$\zeta\in\QQ_{|G|}$ of order not divisible by $p$ to $\zeta^{q}$,
where $q$ is a certain fixed power of $p$, see \cite{Navarro04} and
\cite[Conjecture~9.8]{Navarro18}. This refinement of the McKay
conjecture has now become the McKay-Navarro conjecture, also known
as the Galois-McKay conjecture.

Assume the validity of the MN conjecture. We then have
\[|\{\chi\in \Irr_{p'}(G): \chi^H=\chi\}|= |\{\chi\in \Irr_{p'}(\bN_G(P)): \chi^H=\chi\}|\]
for every subgroup $H$ of $\mathcal{H}$. Since
$\Gal(\QQ_{|G|}/\QQ_{p^\alpha|G|_{p'}})$ is contained in
$\mathcal{H}$ for every nonnegative integer $\alpha\le \log_p|G|_p$,
the conjecture implies that the number of characters at every
$p$-rationality level $\alpha$ in $\Irr_{p'}(G)$ and
$\Irr_{p'}(\bN_G(P))$ is the same:
\[
|\{\chi\in\Irr_{p'}(G): \lev(\chi)=
\alpha\}|=|\{\psi\in\Irr_{p'}(\bN_G(P)): \lev(\psi)= \alpha\}|.
\]

Following \cite{Isaacs-Navarro01}, for each $\alpha\in\ZZ^{+}$, we
denote by $\sigma_{\alpha}$ the automorphism in $Gal(\QQ^{ab}/\QQ)$
that fixes roots of unity of order not divisible by $p$ and maps
every $p$-power root of unity $\xi$ to $\xi^{1+p^\alpha}$. Abusing
notation, when working with a specific group $G$, we also use
$\sigma_{\alpha}$ for the restriction of this automorphism to
$\QQ_{|G|}$.

Let $|G|=p^am$ with $(p,m)=1$ and $a\in\ZZ^+$. Note that
$\sigma_{\alpha}$ is a generator for the cyclic group
$Gal(\QQ_{p^am}/\QQ_{p^\alpha m})$ when $\alpha\geq 2$ or $p>2$.
Note also that there does not exist $\chi\in\Irr_{2'}(G)$ that is
not $2$-rational but is $\sigma_1$-invariant. (Assume so. When
$p=2$, the fixed field of $\sigma_1$ in $\QQ_{p^am}$ is $\QQ_m$ when
$a\leq 2$ and is $\QQ_{m}(\sqrt{2}i)$ when $a\geq 3$. In any case,
we have $\QQ(\chi)\subseteq \QQ_{m}(\sqrt{2}i)$. This contradicts
the main result of \cite{ILNT19} stating that the field of values of
a non-2-rational odd-degree irreducible character always contains
the imaginary unit $i$.) We have shown that, if
$\chi\in\Irr_{p'}(G)$ is not $p$-rational, then

\begin{quote}\emph{$\lev(\chi)$ is precisely the smallest positive integer
$\alpha$ such that $\chi$ is $\sigma_{\alpha}$-invariant.}
\end{quote}

The relevance of $p$-rationality level was noticed even before the
birth of the MN conjecture. In 2001, motivated by Brauer's Problem
12 \cite{Brauer63}, M. Isaacs and G. Navarro \cite{Isaacs-Navarro01}
proposed the following, but in the language of the Galois
automorphism $\sigma_\alpha$.

\begin{conjecture}[Isaacs-Navarro]\label{conj:Isaacs-Navarro}
Let $p$ be a prime, $G$ a finite group and $P\in\Syl_p(G)$. Let
$\alpha\in\ZZ^{+}$. Then $\exp(P/P')\leq p^\alpha$ if and only if
the $p$-rationality level of every $p'$-degree irreducible
characters of $G$ is at most $\alpha$.
\end{conjecture}

\noindent This has been confirmed for $p=2$
\cite{Navarro-Tiep19,Malle19}. In fact, the `if' direction of the
conjecture has been confirmed for all $p$ \cite[Theorem
B]{Navarro-Tiep19} and the `only if' direction has been reduced to
almost quasisimple groups \cite[Theorem C]{Navarro-Tiep19}.

\begin{theorem}[Malle-Navarro-Tiep]\label{thm:Malle-Navarro-Tiep19}
Let $p=2$, $G$ a finite group, and $P\in\Syl_p(G)$. Suppose that
$\exp(P/P')= p^\alpha$. Then
\begin{enumerate}[\rm(i)]
\item If $\alpha=1$ then every $p'$-degree irreducible characters of
$G$ has $p$-rationality level at most $1$.

\item If $\alpha\geq 2$ then the maximal level of an  irreducible $p'$-degree character of
$G$ is $\alpha$.
\end{enumerate}
\end{theorem}

The following simple statement is equivalent to the combination of
Conjectures \ref{conj:continuity1} and \ref{conj:Isaacs-Navarro}, as
shown below.

\begin{conjecture}\label{conj:M<G}
Let $p$ be a prime, $G$ a finite group and $\alpha\in \ZZ^{\geq 2}$.
Let $M$ be a subgroup of $G$ of $p'$-index. Then $G$ has an
irreducible $p'$-degree characters of $p$-rationality level $\alpha$
if and only if $M$ does.
\end{conjecture}

\begin{lemma}\label{lem:1}
Let $p$ be a prime. Let $G$ be a finite group and $N \nor G$. Let
$\theta\in\Irr(N)$ and $\chi\in\Irr(G)$ such that $\theta$ is a
constituent of $\chi_N$.
\begin{itemize}
\item[(i)] Assume that $[G:N]$ is not divisible by $p$ and $\log_{p}|G|_p\geq\alpha\geq
1$. Then $\chi$ is $\sigma_\alpha$-invariant if and only if $\theta$
is. Furthermore, $\chi$ is $\sigma$-invariant for every $\sigma\in
\Gal(\QQ_{|G|}/\QQ_{p^\alpha}|G|_{p'})$ if and only if $\theta$ is.

  \item[(ii)] If $[G:N]$ is not divisible by $p$ and $\max\{\lev(\theta),\lev(\chi)\}\geq
  2$,
then $\lev(\theta)=\lev(\chi)$.
  \item[(iii)] If $p\nmid \chi(1)$ and $\lev(\chi)\geq 1$, then $\lev(\theta)\leq
  \lev(\chi)$.
 \end{itemize}
\end{lemma}

\begin{proof} The first statement of part (i) is \cite[Lemma 4.2]{Navarro-Tiep21} and the second statement is proved similarly.
For parts (ii) and (iii), set $\alpha:=\lev(\chi)$ and
$\beta:=\lev(\theta)$.

We now prove (ii). First suppose that $\alpha\geq 2$. Then $\beta\le
\alpha$ by part (i). Assume by contradiction that $\beta<\alpha$.
Then $\theta$ would be $\sigma$-invariant for every $\sigma\in
\Gal(\QQ_{|G|}/\QQ_{p^{\alpha-1}|G|_{p'}})$, and thus so is $\chi$
by again (i). It would follow that $\lev(\chi)\leq \alpha-1$, a
contradiction. The case $\beta\ge 2$ is similar.

For part (iii) we note that the Galois group
$\Gal(\QQ_{|G|}/\QQ_{p^{\alpha}|G|_{p'}})$ permutes the
$G$-conjugates of $\theta$. Since $p\nmid \chi(1)$, the number of
those conjugates is coprime to $p$. However, the order of
$\Gal(\QQ_{|G|}/\QQ_{p^{\alpha}|G|_{p'}})$ is a $p$-power, and it
follows that some conjugate of $\theta$ is point-wise fixed  by
$\Gal(\QQ_{|G|}/\QQ_{p^{\alpha}|G|_{p'}})$, which implies that
$\theta$ is point-wise fixed by
$\Gal(\QQ_{|G|}/\QQ_{p^{\alpha}|G|_{p'}})$ as well.
\end{proof}

\begin{theorem}\label{thm:McKay-Navarro} The following hold:

\begin{enumerate}[\rm(i)]

\item Conjectures \ref{conj:continuity1} and
\ref{conj:Isaacs-Navarro} imply Conjecture \ref{conj:M<G}.

\item Conjecture \ref{conj:M<G} implies Conjecture
\ref{conj:continuity1}.

\item The only if direction of Conjecture \ref{conj:M<G} implies Conjecture
\ref{conj:Isaacs-Navarro}.

\item Conjectures \ref{conj:continuity1}, \ref{conj:Isaacs-Navarro}, and \ref{conj:M<G} are all consequences
of the MN conjecture.
\end{enumerate}
\end{theorem}

\begin{proof} We begin with proving (i). Assume that $G$ has an irreducible $p'$-degree
character of $p$-rationality level $\alpha\geq 2$ and $M\leq G$ is
of $p'$-index. Let $P\in\Syl_p(M)\subseteq\Syl_p(G)$. Then
$\beta:=\log_p(\exp(P/P'))\geq \alpha$ by Conjecture
\ref{conj:Isaacs-Navarro}. By Conjecture \ref{conj:continuity1}, it
follows that $M$ possesses irreducible $p'$-degree characters of
every level from $2$ to $\beta$, implying that $M$ has one of level
$\alpha$. The if direction of Conjecture \ref{conj:M<G} is argued
similarly, with $G$ and $M$ interchanged.

To prove the remaining statements, we first claim that Conjecture
\ref{conj:continuity1} holds true when $G$ is a $p$-group. Assume
that $G$ has an irreducible $p'$-degree characters of
$p$-rationality level $\alpha\geq 2$. By the fundamental theorem of
finite Abelian groups, $G/G'$ is isomorphic to a direct product of
cyclic groups. The fact that $G/G'$ has an irreducible linear
character of level $\alpha\geq 2$ implies that at least one of those
cyclic direct factors, say $C=\langle x\rangle$, has order at least
$p^\alpha$:
\[|C|=p^e\ge p^\alpha.\] For every $2\leq\beta\leq\alpha$, the
character $x^i\mapsto \zeta^{i\cdot p^{e-\beta}}$ of $C$, where
$\zeta$ is a primitive $p^e$th root of unity, will have
$p$-rationality level precisely equal to $\beta$. We deduce that
$G/G'$ has irreducible $p'$-degree characters of every
$p$-rationality level from $2$ to $\alpha$, as desired.

Now (ii) immediately follows from the claim. Also, (iii) follows
from the claim and \cite[Theorem B]{Navarro-Tiep19}.

The fact that Conjecture \ref{conj:Isaacs-Navarro} follows from the
MN conjecture was already observed before. Therefore, to see (iv),
it is enough to show that Conjecture \ref{conj:continuity1} is a
consequence of the MN conjecture. For this, as remarked at the
beginning of this section, if the MN conjecture holds true then we
may assume that $P\in\Syl_p(G)$ is normal in $G$. Using
Lemma~\ref{lem:1}, we may further assume that $G=P$ is a $p$-group,
but Conjecture \ref{conj:continuity1} in this case is precisely the
claim above.
\end{proof}


\section{$p$-Rationality level of character values}\label{sec:3}

In this section we relate the $p$-rationality level of a character
with the level of the values of the character. We also prove a key
lemma that will be needed in the proof the main results.

\begin{definition} Let $z$ be a complex number belonging the smallest field
containing all roots of unity. The $p$-rationality level of $z$ is
defined as
\[
\lev(z)=\lev_p(z):=\log_p(c(\QQ(z))),
\]
where $c(\QQ(z))$ is the smallest positive integer $f$ such that
$z\in \QQ_f$.
\end{definition}

\begin{lemma}\label{lem:achieved} Let $\chi$ be a character of a finite group $G$. Then
\[
\lev(\chi)=\max_{g\in G}\{\lev(\chi(g))\}.
\]
\end{lemma}

\begin{proof} It is clear that $\lev(\chi(g))\leq \lev(\chi)$ for
every $g\in G$. On the other hand, as $\QQ(\chi)\subseteq
\QQ_{\lcm(c(\QQ(\chi(g))):g\in G)}$, we have
$\lev(\chi)\leq\max_{g\in G}\{\lev(\chi(g))\}$, and the result
follows.
\end{proof}

In view of Lemma \ref{lem:achieved}, we will say that the
$p$-rationality level of $\chi$ is \emph{achieved} at $g\in G$ if
$\lev(\chi)=\lev(\chi(g))$.

\begin{lemma}\label{lem:level of product character}
Let $\varphi, \psi \in\Irr_{p'}(G)$ and $\chi:=\varphi\psi$ such
that $\lev(\varphi)\leq \lev(\chi)$. If the $p$-rationality level of
$\psi$ is achieved at a $p$-element, then
\begin{enumerate}[\rm(i)]
\item  $\lev(\psi)\leq\lev(\chi)$, and

\item either $\lev(\varphi)=\lev(\chi)$ or $\lev(\psi)=\lev(\chi)$.
\end{enumerate}

\noindent In particular, if $G/\Ker(\psi)$ is a $p$-group, then the
same conclusions hold.
\end{lemma}

\begin{proof}
Suppose that the $p$-rationality level of $\psi$ is achieved at some
$p$-element $g\in G$. Since irreducible $p'$-degree characters are
nonzero at $p$-elements (see \cite[Remark 4.1]{DPSS09}), we have
\[
\chi(g)=\varphi(g)\psi(g)\neq 0,
\]
so that $\psi(g)=\chi(g)\varphi(g)^{-1}$. Therefore,
\[
\lev(\psi)=\lev(\psi(g))\leq
\max\{\lev(\chi(g)),\lev(\varphi(g))\}\leq
\max\{\lev(\chi),\lev(\varphi)\}=\lev(\chi),
\]
proving part (i).

For (ii) we note that, if both $\lev(\varphi)$ and $\lev(\psi)$ are
smaller than $\lev(\chi)$, then
\[\QQ(\chi)=\QQ(\varphi\psi)\subseteq
\QQ_{\lcm(c(\varphi),c(\psi))}=\QQ_{p^am}\] for some $a<\lev(\chi)$
and $p\nmid m$, which is a contradiction.

Finally, assume that $G/\Ker(\psi)$ is a $p$-group and again the
level of $\psi$ is achieved at $g\in G$. Since $g\Ker(\psi)$ is a
$p$-element (of $G/\Ker(\psi)$), we have
$g\Ker(\psi)=g_p\Ker(\psi)$. Hence, $\psi(g)=\psi(g_p)$, and it
follows that the level of $\psi$ is achieved at $g_p$, a $p$-element
of $G$.
\end{proof}

It is possible that the assumption on the $p$-rationality level of
$\psi$ being achieved at a $p$-element in Lemma \ref{lem:level of
product character} is superfluous. Suppose that
$\lev(\psi)=\alpha\geq 2$. \cite[Conjecture C]{Navarro-Tiep21}
asserts that \begin{quote} \emph{the degree of the extension
$\QQ_{p^\alpha}/\QQ(\psi_P)$, where $P\in\Syl_p(G)$, is not
divisible by $p$.}\end{quote} This is equivalent to
$\lev(\QQ(\psi_P))=\alpha$, which in turn is equivalent to that
there exists some $g\in P$ such that $\lev(\psi(g))=\alpha$. In
other words, the level of $\psi$ is achieved at the $p$-element $g$.
We rephrase \cite[Conjecture C]{Navarro-Tiep21}.

\begin{conjecture}[Navarro-Tiep]\label{conj:achieved}
The $p$-rationality level higher than 1 of a $p'$-degree irreducible
character is always achieved at a $p$-element.
\end{conjecture}

Let us mention that this rephrasing together with Theorem
\ref{thm:technical} have allowed us to confirm the conjecture for
linear and unitary groups $G$. We in fact have shown that if the
level of $\psi\in\Irr_{p'}(G)$ is achieved at some $g\in G$, then it
is also achieved at $g_p$, where $g_p$ is the $p$-part of $g$ (work
in progress \cite{Hung22}).


\section{2-Rationality of odd-degree
characters}\label{sec:odd-degree}

In this section we prove Theorems \ref{thm:main1} and
\ref{thm:continuity}, using the result of the previous section and
assuming a continuity result (but with automorphisms) for finite
simple groups (Theorem \ref{thm:continuity2}).

We begin with some preliminary results.

\begin{lemma}\label{lem:exp}
Let $N$ be a normal subgroup of $G$ and $p$ a prime. Let $P_N, P_G$,
$P_{G/N}$ be respectively Sylow $p$-subgroups of $N, G$, and $G/N$.
Then
\begin{enumerate}[\rm(i)]

\item $\exp(P_{G/N}/P'_{G/N})\leq \exp(P_G/P'_G)$.

\item $\exp(P_G/P'_G) \leq \exp(P_N/P'_N) \cdot \exp(P_{G/N}/P'_{G/N})$.
\end{enumerate}
\end{lemma}

\begin{proof} We may assume without loss that $P_N\leq P_G$ and
$P_{G/N}=P_G/P_N$. We then have $P_{G/N}/P'_{G/N} \cong P_G/P'_GP_N$
and (i) is clear. To see (ii), note that if $X\nor Y$ then
$\exp(Y)\leq \exp(X)\cdot\exp(Y/X)$. Therefore,
\begin{align*}
\exp(P_G/P'_G) &\leq
\exp\left(\frac{P_G/P'_G}{P_NP'_G/P'_G}\right)\cdot\exp(P_NP'_G/P'_G)\\
&=\exp(P_G/P'_GP_N)\cdot\exp(P_N/(P'_G\cap P_N))\\
&\leq \exp(P_{G/N}/P'_{G/N})\cdot\exp(P_N/P'_N),
\end{align*}
as desired.
\end{proof}

Recall that, for a character $\chi$ of $G$, the determinantal order
$o(\chi)$ of $\chi$ is the order of $\det \chi$ in the group of
linear characters of $G$.

\begin{lemma}\label{lem:canonical extension}
Let $N\nor G$ and $\theta\in\Irr(N)$ is $G$-invariant with
$\gcd(o(\theta)\theta(1),[G:N])=1$. Then $\theta$ has a unique
extension $\widehat{\theta}\in\Irr(G)$ with
$\gcd(o(\widehat{\theta}),[G:N])=1$. Furthermore,
$\QQ(\theta)=\QQ(\widehat{\theta})$.
\end{lemma}

\begin{proof}
This is Corollaries 6.2 and 6.4 of \cite{Navarro18}. See also
\cite[Corollary 6.28]{Isaacs1} for the case where $G/N$ is solvable.
\end{proof}

\begin{lemma}\label{lem:2} Let $N\nor G$. Let $\theta\in\Irr_{p'}(N)$ with $\lev(\theta)=\alpha\geq 2$.
Suppose $\theta$ is $P$-invariant for some $P\in\Syl_p(G)$ and
$p\nmid o(\theta)$. Then there exists $\chi\in\Irr_{p'}(G)$ lying
over $\theta$ such that $\lev(\chi)=\alpha$.
\end{lemma}

\begin{proof}
By Lemma \ref{lem:canonical extension}, $\theta$ has a unique
extension to $NP$, say $\widehat{\theta}\in\Irr_{p'}(NP)$, such that
$p\nmid o(\widehat{\theta})$. Furthermore,
$\QQ(\widehat{\theta})=\QQ(\theta)$. Since
$\Gal(\QQ_{|G|}/\QQ_{p^\alpha |G|_{p'}})$ fixes $\widehat{\theta}$
as well as $\widehat{\theta}^G$, it permutes the irreducible
constituents of $\widehat{\theta}^G$. As
$\widehat{\theta}^G(1)=\theta(1)[G:NP]$ is not divisible by $p$, it
follows that there is some irreducible constituent $\chi$ of
$\widehat{\theta}^G$ that is $\Gal(\QQ_{|G|}/\QQ_{p^\alpha
|G|_{p'}})$-invariant and $p'$-degree. In particular,
$\lev(\chi)\leq \alpha$.

If $\lev(\chi)=0$ then $\chi$ would be $\Gal(\QQ_{|G|}/\QQ_{p
|G|_{p'}})$-fixed, which implies that $\theta$, being a constituent
of $\chi_N$, is $\Gal(\QQ_{|G|}/\QQ_{p |G|_{p'}})$-fixed, violating
the hypothesis that $\lev(\theta)\geq 2$. Therefore we have
$\lev(\chi)\geq 1$. Lemma \ref{lem:1}(iii) now implies that
$\alpha=\lev(\theta)\leq \lev(\chi)$, and hence $\lev(\chi)=\alpha$,
as required.
\end{proof}

\begin{lemma}\label{lem:NTT} Let $P$ be a finite $p$-group acting on a direct
product $G=G_1\times G_2\times\cdots\times G_n$ such that the $G_i$s
are transitively permuted by $P$ and a Sylow $p$-subgroup of $G$ is
invariant under $P$. Let $Q$ be the stabilizer of $G_1$ in $P$ and
$\{x_1,x_2,...,x_t\}$ be a transversal for the right cosets of $Q$
in $P$ with $G_1^{x_i}=G_i$. Then $\theta\in\Irr(G_1)$ is
$Q$-invariant if and only if
$\theta^{x_1}\times\cdots\times\theta^{x_t}$ is $P$-invariant.
\end{lemma}

\begin{proof}
This is a special case of \cite[Lemma 4.1]{Navarro-Tiep-Turull}.
\end{proof}

We need the following continuity result for almost simple groups,
whose proof is delayed until the next section, in order to prove
Theorems \ref{thm:main1} and \ref{thm:continuity}.

\begin{theorem}\label{thm:continuity2}
Let $p=2$, $S$ a finite nonabelian simple group, and $S\nor X\leq
\Aut(S)$ an almost simple group such that $X/S$ is a $p$-group. If
$S$ has an $X$-invariant irreducible $p'$-degree character of
$p$-rationality level $\alpha\geq 2$, then $S$ has $X$-invariant
irreducible $p'$-degree characters of every level from $2$ to
$\alpha$.
\end{theorem}

We can now prove Theorem \ref{thm:continuity}. Note that Theorem
\ref{thm:main1} follows as a consequence of Theorems
\ref{thm:continuity} and \ref{thm:Malle-Navarro-Tiep19}.

\begin{theorem}\label{thm:continuity2-repeated}
Let $p=2$, $G$ a finite group and $P\in\Syl_p(G)$. Suppose that
$\log_p(\exp(P/P'))=\alpha\geq 2$. Then $G$ has irreducible
$p'$-degree characters of every $p$-rationality level from $2$ to
$\alpha$.
\end{theorem}

\begin{proof}
Let $N$ be a minimal normal subgroup of $G$. We may assume that $p$
divides $|N|$ because otherwise we are done by induction on $|G|$.

Suppose first that $N$ is abelian. Then in fact $N\nor P$, so that
$P/N\in\Syl_p(G/N)$. If $\exp(P/P')=\exp((P/N)/(P/N)')$ then we are
done again, and therefore, by Lemma \ref{lem:exp}(i), we may assume
that $\exp(P/P')>\exp((P/N)/(P/N)')$; in particular, $N\nsubseteq
P'$. Since $\exp(N)=p$, using Lemma \ref{lem:exp}(ii), we have
\[
\exp(P/P')=p\cdot\exp((P/N)/(P/N)').
\]
By induction, it is sufficient to show that $G$ has an irreducible
$p'$-degree character of $p$-rationality level precisely equal to
$\log_p(\exp(P/P'))=\alpha$. This is in fact a consequence of
Theorem \ref{thm:Malle-Navarro-Tiep19}.

We may now suppose that $N$ is nonabelian of order divisible by $p$.

We claim that if $NP$ has a $p'$-degree irreducible character of
level $\beta\geq 2$, then so does $G$. Let $H:=NP$ and
$\varphi\in\Irr_{p'}(H)$ such that
\[\lev(\varphi)=\beta.\] We have
$\theta:=\varphi_N$ is irreducible by \cite[Corollary
11.29]{Isaacs1}, and therefore is $H$-invariant. Note that $[H:N]$
is a $p$-power and, as $N$ is perfect, the determinantal order of
$\theta$ is trivial. It follows from Lemma \ref{lem:canonical
extension} that $\theta$ has a unique extension to $H$, say
$\widehat{\theta}$, whose determinantal order is coprime to $p$.
Furthermore, we have $\QQ(\widehat{\theta})=\QQ(\theta)$, and in
particular, \[\lev(\widehat{\theta})=\lev(\theta).\]

Using Gallagher's lemma, we have $\varphi=\widehat{\theta}\rho$ for
some linear character $\rho$ of the $p$-group $H/N$. By Lemma
\ref{lem:1}(iii), we know that $\lev(\theta)\leq
\lev(\varphi)=\beta$, and whence $\lev(\widehat{\theta})\leq \beta$.
Using Lemma \ref{lem:level of product character}, we deduce that
$\lev(\rho)\leq\beta$ and
\[
\text{either } \lev(\widehat{\theta})=\beta \text{ or }
\lev(\rho)=\beta.
\]
In the former case, we have $\lev({\theta})=\beta$ and Lemma
\ref{lem:2} therefore guarantees that $G$ has a $p'$-degree
irreducible character of level $\beta$. Suppose the latter case that
$\lev(\rho)=\beta$. Then it is easy to see that \[\beta\leq
\log_p(\exp((H/N)/(H/N)'))=:\gamma.\] Note that $H/N\in\Syl_p(G/N)$.
So $G/N$ has $p'$-degree irreducible characters of all levels from
$2$ to $\gamma$ by induction, implying in particular that $G/N$, and
therefore $G$, has a $p'$-degree irreducible character of level
$\beta$. The claim is proved.

It is now sufficient to show that $H=NP$ has irreducible $p'$-degree
characters of every $p$-rationality level from $2$ to $\alpha$.
Replacing $N$ by the direct product of the simple factors in an
orbit of the action of $P$ on the factors of $N$ and repeating the
above arguments if necessary, we may furthermore assume that $P$
transitively permutes the factors of $N$.

Let $S$ be a simple factor of $N$. Let $K:=\bN_H(S)$ and
$C:=\bC_H(S)$. Then $K/C$ is an almost simple group with socle
$SC/C\cong S$. Moreover, since $N\subseteq C$, $K/SC$ is a
$p$-group.

We know from Theorem \ref{thm:Malle-Navarro-Tiep19} that $H$ has a
$p'$-degree irreducible character of level
$\alpha=\log_p(\exp(P/P'))$, say $\chi$. As above we have
$\eta:=\chi_N$ is irreducible and $\chi=\widehat{\eta}\sigma$, where
$\widehat{\eta}\in\Irr_{p'}(H)$ is the extension of $\eta$ to $H$
produced by Lemma \ref{lem:canonical extension} with
$\lev(\eta)=\lev(\widehat{\eta})$, and $\sigma$ is a linear
character of $H/N=P/(P\cap N)$. Moreover,
\[\lev(\chi)=\alpha\in\{\lev(\eta),\lev(\sigma)\}.\]

If $\lev(\sigma)=\alpha$ then
$\exp(P/P')=p^\alpha=p^{\lev(\sigma)}\leq \exp(P/P'(P\cap N))$,
implying that $\exp(P/P')=\exp(P/P'(P\cap N))=p^\alpha$ by Lemma
\ref{lem:exp}. It is then easy to see, as seen in the proof of
Theorem \ref{thm:McKay-Navarro}, that $P/(P\cap N)$, and thus $H$,
has $p'$-degree irreducible characters of all levels from $2$ to
$\alpha$, as desired.

So we may assume that $\lev(\eta)=\alpha$. Note that
$\eta\in\Irr(N)$ is a product $\eta=\eta_1\times \cdots \times
\eta_n$ of certain characters $\eta_i\in\Irr_{p'}(S^{g_i})$, and so
there must be some $\eta_i$ of level $\alpha$ (and the others will
have level $\leq \alpha$). Without loss, we assume that
$\lev(\eta_1)=\alpha$. Since $\eta$ is $H$-invariant, $\eta_1$ is
$K$-invariant. Now viewing $\eta_1$ as a character of $SC/C\cong S$,
we see that $\eta_1$ is invariant under $K/C$.

We now can apply Theorem \ref{thm:continuity2} to the almost simple
group $K/C$ with socle $SC/C$ and the $K/C$-invariant character
$\eta_1\in\Irr_{p'}(SC/C)$ to deduce that $SC/C$ possesses
$K/C$-invariant irreducible $p'$-degree characters $\theta_i$ of
every level $2\leq i\leq\alpha$. Viewing $\theta_i$ as a character
of $SC$, we have that $\vartheta_i:={\theta_i}_S\in\Irr_{p'}(S)$ is
$K$-invariant.

We partly follow the setup in the proof of \cite[Theorem
2.4]{Navarro-Tiep19}. In particular, let $Q:=\bN_P(S)=K\cap
P\in\Syl_p(K)$ and $\{1=g_1,g_2,...,g_n\}$ be a transversal for the
right cosets of $Q$ in $P$ with $N=S^{g_1}\times\cdots\times
S^{g_n}$. Note that $QC/C$ is a Sylow $p$-subgroup of $K/C$. Since
$\vartheta_i$s are $K$-invariant, they are $Q$-invariant as well. By
Lemma \ref{lem:NTT}, these $\vartheta_i$s give rise to the
characters
\[
\psi_i:=\vartheta_i^{g_1}\times \cdots\times
\vartheta_i^{g_n}\in\Irr_{p'}(N)
\]
that are $P$-invariant. It follows from Lemma \ref{lem:canonical
extension} that $\psi_i$ extends to some
$\widehat{\psi}_i\in\Irr_{p'}(H)$ with
$\lev(\widehat{\psi}_i)=\lev(\psi_i)$. As
$\lev(\psi_i)=\lev(\vartheta_i)=\lev(\theta_i)=i$, we have
$\lev(\widehat{\psi}_i)=i$. The characters $\widehat{\psi}_i$ with
$2\leq i\leq \alpha$ fulfill our requirement, and the proof is
finished.
\end{proof}


\section{Simple groups: some generalities}\label{sec:simple-groups}

In this and the next two sections, we will prove Theorem
\ref{thm:continuity2}, which is the case $p=2$ of the following:

\begin{conjecture}\label{conj:continuity2-repeated}
Let $p$ be a prime, $S$ a finite nonabelian simple group and $S\nor
X\leq \Aut(S)$ an almost simple group such that $X/S$ is a
$p$-group. If $S$ has an $X$-invariant irreducible $p'$-degree
character of $p$-rationality level $\alpha\geq 2$, then $S$ has
$X$-invariant irreducible $p'$-degree characters of every level from
$2$ to $\alpha$.
\end{conjecture}

\begin{remark} Almost identical arguments as in the proof of
Theorem \ref{thm:continuity2-repeated} show that
Conjecture~\ref{conj:continuity1} follows from Conjecture
\ref{conj:continuity2-repeated} and the `only if' direction of
Conjecture \ref{conj:Isaacs-Navarro}. This direction, as we
mentioned already, has been reduced to almost quasisimple groups,
and the exact statement one has to verify for those groups is in
\cite[Conjecture 5.4]{Navarro-Tiep19}. Therefore, Conjecture
\ref{conj:continuity1} has been reduced to almost quasisimple
groups, as follows:
\end{remark}

\begin{theorem}\label{thm:reduction}
Conjecture \ref{conj:continuity1} follows from Conjecture
\ref{conj:continuity2-repeated} and \cite[Conjecture
5.4]{Navarro-Tiep19}.
\end{theorem}

\begin{proposition}
Conjecture \ref{conj:continuity2-repeated} holds true when $S$ is a
sporadic simple group, the Tits group, an alternating group, or a
simple group of Lie type in characteristic $p$.
\end{proposition}

\begin{proof}
The statement is trivial when all the irreducible $p'$-degree
characters of $S$ have $p$-rationality level at most $2$. We will
see that this is indeed the case for all the groups in question. In
fact, it was already explained in the proofs of \cite[Propositions
2.1, 2.2, 2.3, and 2.4]{Malle19} that, when $S\neq {}^2F_4(2)'$,
every irreducible $p'$-degree character of a group in consideration
is almost $p$-rational. (We note a typo in the proof of
\cite[Proposition 2.4]{Malle19}: $\sigma_e$ should be replaced by
$\sigma_1$.) The same thing happens with the Tits group
${}^2F_4(2)'$ and $p$ odd. The group ${}^2F_4(2)'$ does have some
odd-degree irreducible characters of $2$-rationality level $2$
though, but none of level higher than $2$, and thus the statement
still holds in this case.
\end{proof}

The focus of Conjecture \ref{conj:continuity2-repeated} is therefore
on simple groups of Lie type in non-defining characteristic. To
continue, we need to setup a general framework for these groups via
algebraic groups and the Deligne-Lusztig theory on which we refer
the reader to \cite{Carter85,Digne-Michel91} for some background.

Let $\bG$ be a connected reductive group defined over an
algebraically closed field of characteristic $r>0$ and
$F:\bG\rightarrow \bG$ is a Frobenius endomorphism that defines an
$\mathbb{F}_q$-rational structure $G=\bG^F$, where $q$ is a power of
$r$. Suppose that $(\bG^\ast,F^\ast)$ is dual to $(\bG,F)$ and let
$G^\ast:={\bG^\ast}^{F^\ast}$.

The set $\Irr(G)$ is naturally partitioned into the (rational)
Lusztig series $\EC(G,s)$ associated to various $G^\ast$-conjugacy
classes of semisimple elements $s\in G^\ast$. The series $\EC(G,s)$
is defined to be the set of irreducible characters of $\bG^F$
occurring in some Deligne-Lusztig character $R_{\bT}^{\bG}\theta$,
where $\bT$ is an $F$-stable maximal torus of $\bG$ and
$\theta\in\Irr(\bT^F)$ such that the geometric conjugacy class of
$(\bT,\theta)$ corresponds to the $G^\ast$-conjugacy class
containing $s$ (see \cite[Theorem 4.4.6]{Carter85} for this
correspondence). For later use, we also note that, when $\bZ(\bG)$
is connected, there is a natural bijection $\pi_s$ from $\EC(G,s)$
to $\EC(\bC_{G^\ast}(s),1)$ -- the set of unipotent characters of
$\bC_{G^\ast}(s)$ (see \cite[Theorem 4.7.1]{GM20} for the
description of $\pi_s$). If $\chi\in\EC(G,s)$ corresponds to
$\psi\in\EC(\bC_{G^\ast}(s),1)$ under this bijection, we will say
that $(s,\psi)$ is the Jordan decomposition of $\chi$.

Though the actions of $\sigma_\alpha$ and $\Aut(G)$ on $\Irr(G)$ are
generally difficult to control, it turns out that their effect on
Lusztig series $\EC(G,s)$ is nicely behaved. (Recall that
$\sigma_{\alpha}\in Gal(\QQ^{ab}/\QQ)$ fixes roots of unity of order
not divisible by $p$ and maps every $p$-power root of unity $\xi$ to
$\xi^{1+p^\alpha}$.) In fact, by a result of A.\,A. Schaeffer Fry
and J. Taylor \cite[Lemma 3.4]{ST18}, we have
\begin{equation}\label{eq:1}
\EC(G,s)^{\sigma_\alpha}=\EC(G,s_p^{1+p^\alpha}\cdot s_{p'}),
\end{equation}
where $s_p$ and $s_{p'}$ are respectively the $p$- and $p'$-parts of
$s$.

To describe the action of $\Aut(G)$ on Lusztig series, we need some
more notation. Let $\varphi\in \Aut(G)$. It is well-known that
$\varphi$ is then the restriction of some bijective morphism $\tau$
of $\bG$ that commutes with $F$. This $\tau$ induces a bijective
morphism $\tau^\ast$ on $\bG^\ast$ that commutes with $F^\ast$,
which implies that the restriction
$\varphi^\ast:=\tau^\ast\hspace{-5pt}\downarrow_{G^\ast}$ is an
automorphism of $G^\ast$. We call $\varphi^\ast$ an automorphism
dual to $\varphi$. By \cite[Corollary 2.4]{Navarro-Tiep-Turull}
(this is due to C. Bonnaf\'{e} and a more detailed proof was later
given by J. Taylor in \cite[Proposition 7.2]{Taylor18}), we have
\begin{equation}\label{eq:2}
\EC(G,s)^{\varphi}=\EC(G,\varphi^\ast(s)).
\end{equation}

\begin{proposition}\label{prop:technical-0}
Let $G$ and $G^\ast$ be as above, and $r$ their defining
characteristic. Let $p\neq r$ a prime and $\chi\in\Irr_{p'}(G)$ with
$\lev(\chi)=\alpha\geq 1$. Suppose that $\chi\in\EC(G,s)$ for some
semisimple element $s\in G^\ast$. Then
\begin{enumerate}[\rm(i)]
\item $s$ is $p$-central in $G^\ast$, and

\item $\ord(s_p)\leq p^\alpha$, where $s_p$ is the $p$-part of
$s$.

\end{enumerate}
\end{proposition}

\begin{proof}
By \cite[Proposition 7.2]{Malle07}, the fact that $\chi\in \EC(G,s)$
has $p'$-degree implies that $s$ is a $p$-central element of
$G^\ast$. Since $\lev(\chi)=\alpha$, $\chi$ is
$\sigma_\alpha$-invariant, and so the Lusztig series $\EC(G,s)$ is
$\sigma_\alpha$-invariant. It follows that, by (\ref{eq:1}), $s$ is
conjugate to $s_p^{1+p^\alpha}\cdot s_{p'}$ in $G^\ast$. In
particular, $s_p$ is conjugate to $s_p^{1+p^\alpha}$ in $G^\ast$ as
well.

Recall that $s$ is $p$-central. Therefore so is $s_p$, and it
follows that $\bC_{G^\ast}(s_p)$ contains a Sylow $p$-subgroup, say
$P^\ast$, of $G^\ast$. Since $\bN_{G^\ast}(P^\ast)$ controls
$G^\ast$-fusion in $\bC_{G^\ast}(P^\ast)$ (see for instance
\cite[Lemma 5.12]{Isaacs08}), we deduce that
$s_p^{1+p^\alpha}=x^{-1}s_px$ for some $x\in \bN_{G^\ast}(P^\ast)$.
Let $k:=|x|_{p'}$ and $y:=x^k\in P^\ast$. We then have
$s_p^{(1+p^\alpha)^k}=y^{-1}s_py=s_p$ since
$s_p\in\bC_{G^\ast}(P^\ast)$. It follows that $|s_p|$ divides
$(1+p^\alpha)^k-1$, which yields $|s_p|$ divides $p^\alpha$ because
$p\nmid k$, as desired.
\end{proof}

In Proposition \ref{prop:technical-0}, generally $\log_p(\ord(s_p))$
could be smaller than $\lev(\chi)$. However, in certain nice
situations as in the following key results that are sufficient for
our purpose, they are indeed equal. The next two theorems might be
useful in future work on Conjecture~\ref{conj:continuity2-repeated}
for odd $p$.

\begin{theorem}\label{thm:technical}
Let $\alpha\in \ZZ^{\geq 2}$. Let $\bG$, $\bG^\ast$, $G$, $G^\ast$,
$F$, and $F^\ast$ be as above, and $r$ the defining characteristic
of $\bG$. Let $p\neq r$ a prime and $\chi\in\Irr_{p'}(G)$. Suppose
that ${\chi}\in\EC({G},s)$ for some ($p$-central) semisimple element
$s\in G^\ast$ such that $\bL^\ast:=\bC_{\bG^\ast}(s)$ is a Levi
subgroup of $\bG^\ast$ and all the $p'$-degree (unipotent)
characters in $\EC(\bL^F,1)$, where $\bL$ is an $F$-stable Levi
subgroup of $\bG$ dual to $\bL^\ast$, have $p$-rationality level at
most $\alpha-1$. We have
\begin{enumerate}[\rm(i)]
\item $\lev({\chi})=\alpha$ if and only if $\ord(s_p)= p^\alpha$,
and in such case,

\item the $p$-rationality levels of all the $p'$-degree characters in
$\EC({G},s)$ are the same.
\end{enumerate}
\end{theorem}

\begin{proof} Of course (ii) follows from (i). For (i), using Proposition \ref{prop:technical-0}, we just need to
show that if $\ord(s_p)= p^\beta$ then $\lev({\chi})\leq\beta$, for
$\beta\in\{\alpha,\alpha-1\}$. Equivalently, we wish to show that if
$\ord(s_p)= p^\beta$ then $\chi$ is $\sigma_\beta$-fixed.

Since $\bL^\ast$ is a Levi subgroup of $\bG^\ast$ and $s\in
\bZ(\bL^\ast)^{F^\ast}$, according to \cite[Proposition
8.26]{Cabanes-book}, the multiplication by $\widehat{s}$, a certain
linear character of $\Irr(\bL^F)$ naturally defined by $s$ (see
\cite[(8.19)]{Cabanes-book}), induces a bijection between
$\mathcal{E}(\bL^F,1)$ and $\mathcal{E}(\bL^F,s)$. Note that the
correspondence $s\mapsto \widehat{s}$ from $\bZ(\bL^\ast)^{F^\ast}$
to the linear characters of $\Irr(\bL^F)$ is a group isomorphism.
Therefore,
\[
\ord(s_p)=\ord(\widehat{s})_p=p^{\lev(\widehat{s})},
\]
and it follows that \[ \widehat{s} \text { is }
\sigma_\beta-\text{fixed}.
\]
This in turn implies that \[\text{every } p'-\text{degree character
} \psi\in \mathcal{E}(\bL^F,s) \text{ is } \sigma_\beta-\text{fixed}
\]
since $\psi$ is of the from $\lambda \widehat{s}$, where $\lambda$
is a $p'$-degree character in $\mathcal{E}(\bL^F,1)$, which is
$\sigma_\beta$-fixed by the hypothesis.

On the other hand, by \cite[Proposition 8.27]{Cabanes-book}, the
Lusztig induction functor $\mathbf{R}_{\bL}^{\bG}$ induces a
bijection
\[
\varepsilon_{\bG} \varepsilon_{\bL}\mathbf{R}_{\bL}^{\bG}:
\EC(\bL^F,s)\rightarrow \EC(\bG^F,s),
\]
where $\varepsilon_{\bG}:=(-1)^{\sigma(\bG)}$ with $\sigma(\bG)$ the
$\mathbb{F}_q$-rank of $\bG$. Suppose that \[\chi= \varepsilon_{\bG}
\varepsilon_{\bL}\mathbf{R}_{\bL}^{\bG}(\psi)\] for some $\psi\in
\EC(\bL^F,s)$, which must be $p'$-degree since $\chi$ is. By the
formula for $\mathbf{R}_{\bL}^{\bG}$ (see \cite[Proposition
12.2]{Digne-Michel91}), each value of $\chi$ is a linear combination
with integer coefficients of values of $\psi$. Since $\psi$ is
$\sigma_\beta$-fixed, we deduce that $\chi$ is $\sigma_\beta$-fixed,
as desired.
\end{proof}

In Theorem \ref{thm:technical}, the assumption on the centralizer
$\bC_{\bG^\ast}(s)$ being Levi can be dropped in the case $\bG$ has
connected center, by using our Proporsition \ref{prop:technical-0}
and a result of Srinivasan and Vinroot \cite{SV20} on the effect of
Galois automorphisms on irreducible characters of $G:=\bG^F$ via
their Jordan decomposition. We thank the referee for bringing this
discussion to our attention.

Assume that $\bG$ is a connected reductive group with connected
center and $F$ a suitable Frobenius endomorphism. Let $\mathbf{e}$
be the exponent of $G$ (which turns out to be the same as the
exponent of the dual group $G^\ast$) and $\sigma\in
Gal(\QQ_\mathbf{e}/\QQ)$ such that $\sigma(e^{2\pi
i/\mathbf{e}})=e^{2\pi i r/\mathbf{e}}$ for some $r$ coprime to
$\mathbf{e}$. The main result of \cite{SV20} asserts that if
$\chi\in\Irr(G)$ has Jordan decomposition $(s,\psi)$, then
$\chi^{\sigma}$ has Jordan decomposition $(s^r,\psi^\sigma)$. (In
fact, according to the remark in page 18 of the \emph{loc. cit.},
this appears to hold true even when the hypothesis is relaxed to
$\bC_{\bG^\ast}(s)$ being connected. The statement, however, was not
formally proved yet.)

\begin{theorem}\label{thm:extra}
Let $\alpha\geq 2$, $\bG$, $G$, $G^\ast$, $p$ and $\chi$ as in
Theorem \ref{thm:technical}. Suppose that $\bG$ has connected
center, ${\chi}\in\EC({G},s)$ for some ($p$-central) semisimple
element $s\in G^\ast$, and all the $p'$-degree unipotent characters
of $\bC_{G^\ast}(s)$ have $p$-rationality level at most $\alpha-1$.
Then $\lev({\chi})=\alpha$ if and only if $\ord(s_p)= p^\alpha$.
\end{theorem}

\begin{proof} For the `only if' implication, assume that $\lev({\chi})=\alpha\geq 2$.
Then it follows from Proporsition \ref{prop:technical-0} that
$\ord(s_p)\leq p^\alpha$. Assume to the contrary that $\ord(s_p)\leq
p^{\alpha-1}$. Then, by (\ref{eq:1}), the Lusztig series
$\EC({G},s)$ is $\sigma_{\alpha-1}$-fixed. The result of \cite{SV20}
described above and the assumption on the $p'$-degree unipotent
characters of $\bC_{G^\ast}(s)$ having $p$-rationality level at most
$\alpha-1$ then imply that $\lev(\alpha)\leq \alpha-1$, a
contradiction. The same arguments also show that if $\ord(s_p)=
p^\alpha$ then $\lev({\chi})=\alpha$, and thus the proof is
complete.
\end{proof}

We will also need the following well-known result.

\begin{lemma}\label{lem:restriction-center}
Let $s\in G^\ast$ be a semisimple element.
\begin{enumerate}[\rm(i)]
\item $\chi\hspace{-3pt}\downarrow _{\bZ(G)}$ is the same for all
$\chi\in\EC(G,s)$. In fact, $\chi\hspace{-3pt}\downarrow
_{\bZ(G)}=\chi(1)\cdot \theta\hspace{-3pt}\downarrow _{\bZ(G)}$ for
all $\chi\in\EC(G,s)$ and all pairs $(\bT,\theta)$ in the geometric
conjugacy class determined by $s$.

\item Assume that $|\bZ(G)|=|G^\ast/(G^\ast)'|$. Then all characters
in $\EC(G,s)$ restricts trivially to $\bZ(G)$ if and only if $s\in
(G^\ast)'$.
\end{enumerate}
\end{lemma}

\begin{proof}
(i) follows from \cite[Lemma 2.2]{Malle07} and the `if' direction of
(ii) is \cite[Lemma 4.4(ii)]{Navarro-Tiep13}.

Suppose that all characters in $\EC(G,s)$ restricts trivially to
$\bZ(G)$. Let $\bT^\ast$ be an $F^\ast$-stable maximal torus of
$\bG^\ast$ such that $s\in (\bT^\ast)^{F^\ast}$. Let $\bT$ be the
$F$-stable maximal torus of $\bG$ in duality with $\bT^\ast$. The
duality map then gives rise to an isomorphism $\Delta$ between
$(\bT^\ast)^{F^\ast}$ and the character group $\Irr(\bT^F)$ (see
\cite[Proposition 4.4.1]{Carter85}). Suppose that $\theta\in
\Irr(\bT^F)$ corresponds to $s$ under $\Delta$. Then, by the
definition of $\Delta$, $(\bT,\theta)$ is in the geometric conjugacy
class determined by $s$. Using (i), we know that $\theta$ is trivial
on $\bZ(G)$.

Under our assumption $|\bZ(G)|=|G^\ast/(G^\ast)'|$, the isomorphism
$\Delta$ restricts to an isomorphism between
$(\bT^\ast)^{F^\ast}\cap (G^\ast)'$ and the group of those
characters in $\Irr(\bT^F)$ that are trivial on $\bZ(G)$ (see the
proof of \cite[Lemma 4.4]{Navarro-Tiep13}). Therefore, the
conclusion of the previous paragraph implies that $s\in (G^\ast)'$.
\end{proof}

We end this section by solving Conjecture
\ref{conj:continuity2-repeated} when $p=2$ for groups not of type
$A$ and $E_6$. These two types will be handled in the next two
sections.

In what follows, we use $\epsilon$ for either $\pm$ or $\pm1$,
depending on the context. We also use $PSL^-$ for $PSU$ and $E_6^-$
for ${}^2E_6$.

\begin{proposition}\label{prop:linhtinh}
Conjecture \ref{conj:continuity2-repeated} holds true when $p=2$ and
$S$ is one of the following groups in odd characteristic:
\begin{enumerate}[\rm(i)]
\item $PSL^\epsilon_n(q)$, where $q\equiv -\epsilon\,(\bmod\,4)$, or
$q\equiv \epsilon\,(\bmod\,4)$ and $n$ is a $2$-power.

\item $E_6^\epsilon(q)$, where $q\equiv -\epsilon\,(\bmod\,4)$.

\item $PSp_{2n}(q), P\Omega_{2n+1}(q), P\Omega_{2n}^\pm(q), G_2(q), F_4(q), E_7(q), E_8(q), {}^2G_2(q),
{}^3D_4(q)$.
\end{enumerate}
\end{proposition}

\begin{proof}
The commutator factor group of a Sylow $2$-subgroup of these groups
is elementary abelian. (See Sections 3 and 4 of
\cite{Navarro-Tiep16} for the proof or \cite[Proof of Proposition
3.2]{Malle19} for a simplified proof.) Therefore, by Theorem
\ref{thm:Malle-Navarro-Tiep19}, every $p'$-degree irreducible
character of $S$ has $2$-rationality level at most $1$.
\end{proof}


\section{Linear and unitary groups}\label{sec:linear-unitary}

In this section we prove Theorem \ref{thm:continuity2} for
$S=PSL_n^\epsilon(q)$.

\subsection{$p$-Rationality level of $p'$-degree characters of $GL^\epsilon_n(q)$ and $SL^\epsilon_n(q)$}
Using Theorem \ref{thm:technical}, we can control the
$p$-rationality level of $p'$-degree characters of
$GL^\epsilon_n(n)$ or their projective versions. However, as we will
see in this subsection, it is not so easy to do the same for
$SL_n^\epsilon$ and $PSL_n^\epsilon(q)$.

\begin{proposition}\label{prop:technical}
Let $\alpha\in\ZZ^{\geq 2}$. Let $\widetilde{G}=PGL^\epsilon_n(q)$
or $GL_n^\epsilon(q)$, $p \nmid q$ a prime and
$\widetilde{\chi}\in\Irr_{p'}(\widetilde{G})$. Suppose that
$\widetilde{\chi}\in\EC(\widetilde{G},s)$ for some semisimple
element $s\in SL_n^\epsilon(q)$ or $GL_n^\epsilon(q)$, respectively.
Then $\lev(\widetilde{\chi})=\alpha$ if and only if $\ord(s_p)=
p^\alpha$.
\end{proposition}

\begin{proof}
Suppose that $\bG:=GL_n(\overline{\FF_r})$, where $r$ is a prime
divisor of $q$, and $F$ acts on $\bG$ by raising each matrix entry
to the $q$-power for the linear case and further taking transpose
and inverse for the unitary case, so that $\bG^F=GL^\epsilon_n(q)$.
Note that $\bG=\bG^\ast$ is self dual. It is well-known that, for
each semisimple element $s\in \bG^F$, the centralizer $\bC_{\bG}(s)$
is a Levi subgroup consisting of block-diagonal matrices. Such a
Levi subgroup is $F$-stable and its fixed point group
$\bC_{\bG^F}(s)$ is isomorphic to a product of some suitable
subgroups $GL_{n_i}(q^{k_i})$ or $GU_{n_i}(q^{k_i})$ (see
\cite{Carter81,FS89}; they are all $GL$ when $\epsilon=+$ but
possibly a mix of $GL$ and $GU$ when $\epsilon=-$). Unipotent
characters of these subgroups are always rational-valued (see
\cite[Corollary 1.12]{Lusztig02} and the comment right after it).
Theorem \ref{thm:technical} (or Theorem \ref{thm:extra}) now can be
applied for the situation in consideration. Once the result is
verified for $GL^\epsilon$, it follows for $PGL^\epsilon$ by an
inflation argument.
\end{proof}

\begin{proposition}\label{prop:technical2}
Let $G:=SL_n^\epsilon(q)\triangleleft
GL_n^\epsilon(q):=\widetilde{G}$, $p \nmid q$ a prime and
$\chi\in\Irr_{p'}(G)$ with $\lev(\chi)\geq 2$. Suppose that $\chi$
is invariant under $\bO_p(\widetilde{G}/G)$. Then
\begin{enumerate}[\rm(i)]
\item $\lev(\chi)=\min\{\lev(\widetilde{\chi}):\widetilde{\chi}\in
\Irr(\widetilde{G}|\chi)\}$, where $\Irr(\widetilde{G}|\chi)$ is the
set of irreducible characters of $\widetilde{G}$ lying above $\chi$.

\item Suppose that $\widetilde{\chi}\in \Irr(\widetilde{G}|\chi)$
and $\widetilde{\chi}$ belongs to the Lusztig series associated to a
semisimple element $s\in \widetilde{G}$. We have
\[
\lev({\chi})=\min_{t\in \bZ(\widetilde{G})}\{\log_p(\ord(t s)_p)\}.
\]
\end{enumerate}
\end{proposition}

\begin{proof}
(i) Let $G \triangleleft M \triangleleft \widetilde{G}$ such that
$M/G = \bO_p(\widetilde{G}/G)$. By the assumption, $\chi$ is
$M$-invariant. By Lemma \ref{lem:canonical extension}, $\chi$
extends to some $\psi\in\Irr(M)$ with $\QQ(\chi)=\QQ(\psi)$. It
follows that $\lev(\chi)=\lev(\widetilde{\chi})$ for any
$\widetilde{\chi}\in\Irr(\widetilde{G})$ lying above $\psi$, by
Lemma \ref{lem:1}(iii). The result then follows by Lemma
\ref{lem:1}(iv).

\medskip

(ii) Note that the restrictions of all characters in
$\Irr(\widetilde{G}|\chi)$ to $G$ are the same. It is well-known
that if $\widetilde{\chi}_1\in\Irr(\widetilde{G})$ such that
$\widetilde{\chi}_1\hspace{-3pt}\downarrow_G=\widetilde{\chi}\hspace{-3pt}\downarrow_G$,
then $\widetilde{\chi}_1\in\EC(\widetilde{G},ts)$ for some $t\in
\bZ(\widetilde{G})$ (see, for instance, \cite[Lemma
4.1]{Kleshchev-Tiep09}). Conversely, if
$\widetilde{\chi}_1\in\EC(\widetilde{G},ts)$ is labeled by the same
unipotent character of
$\bC_{\widetilde{G}}(ts)=\bC_{\widetilde{G}}(s)$ labelling
$\widetilde{\chi}$, then
$\widetilde{\chi}_1\hspace{-3pt}\downarrow_G=\widetilde{\chi}\hspace{-3pt}\downarrow_G$.
The result follows by part (i) and Proposition \ref{prop:technical}.
\end{proof}

\subsection{Parametrization of odd-degree
characters of $GL_n^\epsilon(q)$} For the rest of this section we
assume that $p=2$ and $q$ is an odd prime power. We will use the
parametrization of odd-degree irreducible characters of
$GL_n^\epsilon(q)$ due to E. Giannelli, A. Kleshchev, G. Navarro,
and P.\,H. Tiep \cite[\S2 and \S5]{GKNT17}. This in turn is based on
the well-known Dipper-James parametrization of complex irreducible
characters of $GL_n(q)$.

We need some notation to describe the mentioned parametrization. Let
$n\in\ZZ^+$. The $2$-adic expansion of $n$ is a sum of the form
\[n=\sum_{i=0}^\infty a_i2^{i} \text{ with } a_i\in\{0,1\}.\] Each positive
term in the sum is called a $2$-adic part of $n$. We call an integer
decomposition of the form $n=n_1+n_2+\cdots+n_k$ such that $k\geq
1$, $n_1>n_2>\cdots>n_k\geq 1$ and every $2$-adic part of $n$ is a
$2$-adic part of some $n_i$ a \emph{proper decomposition} of $n$.

Fix $\epsilon\in\{\pm 1\}$. Let $C_{q-\epsilon}$ denote the cyclic
subgroup of order $q-\epsilon$ of $\FF_{q^2}^\times$ and
$\widetilde{C}_{q-\epsilon}$ denote the character group of
$C_{q-\epsilon}$. Let $\zeta$ be a generator of $C_{q-\epsilon}$ and
$\widetilde{\zeta}$ be a $(q-\epsilon)$-th primitive root of unity
in $\CC$. Then there is a bijection $s\mapsto \widehat{s}$ from
$C_{q-\epsilon}$ to $\widetilde{C}_{q-\epsilon}$ defined by
$\widehat{s}(\zeta)=\widetilde{\zeta}^i$, where $s=\zeta^i$.

For a partition $\lambda$ of $n\in\ZZ^+$, let $\varphi^\lambda$
denote the unipotent character of $GL^\epsilon_n(q)$ labeled by
$\lambda$ (see \cite[\S 13.8]{Carter85}). Let $S(s,\lambda)$ denote
the irreducible character of $GL^\epsilon_n(q)$ labeled by the
(semisimple and central) element $\diag(s,...,s)\in
GL_n^\epsilon(q)$ and the unipotent character $\varphi^\lambda$
under the Jordan decomposition of characters mentioned in the
previous section; that is, by \cite[Proposition
13.30]{Digne-Michel91},
\[
S(s,\lambda)=\varphi^\lambda(\widehat{s}\diamond \det).
\]
where $\det$ denotes the determinant function and $\diamond$ is just
the composition.

The following lemma is the combination of Lemmas 2.5 and 5.2 of
\cite{GKNT17}. Here $\circ$ denotes the Lusztig induction (note that
it is indeed the Harish-Chandra induction in the case $\epsilon=1$).

\begin{lemma}\label{lem:GKNT}
Let $q$ be an odd prime power. Every odd-degree irreducible
character $\widetilde{\chi}$ of $GL_n^\epsilon(q)$ can be uniquely
written in the form
\begin{equation}\label{eq:3}
\widetilde{\chi}=S(s_1,\lambda_1)\circ
S(s_2,\lambda_2)\circ\cdots\circ S(s_k,\lambda_k),
\end{equation}
where $n=n_1+n_2+\cdots+n_k$ is a proper decomposition of $n$ and,
for $1\leq i\leq k$, $s_i\in C_{q-\epsilon}$  are pairwise different
and $\lambda_i$ is partition of $n_i$ such that
$\varphi^{\lambda_i}(1)\in\Irr_{2'}(GL_{n_i}^\epsilon(q))$.
\end{lemma}

We also need the following.

\begin{lemma}\label{lem:abc}
Let $q$ be an odd prime power. Then every
$\widetilde{\chi}\in\Irr_{2'}(GL^\epsilon_n(q))$ restricts
irreducibly to $SL_n^\epsilon(q)$. Furthermore, if $n$ is not a
$2$-power, then every ${\chi}\in\Irr_{2'}(SL^\epsilon_n(q))$ is
extendible to $GL^\epsilon_n(q)$.
\end{lemma}

\begin{proof} This follows from \cite[Lemma 10.2]{ST18}.
\end{proof}

\subsection{Proof of Theorem \ref{thm:continuity2} for $PSL^\epsilon_n(q)$}

\begin{proposition}\label{prop:SL}
Conjecture \ref{conj:continuity2-repeated} holds true for $p=2$ and
$S=PSL^\epsilon_n(q)$ with $q$ odd and $n$ is not a $2$-power.
\end{proposition}

\begin{proof} Suppose that $\chi\in\Irr_{2'}(S)$ with $\lev(\chi)=\alpha\geq
2$ and $\chi$ being $X$-invariant. There is nothing to prove if
$\alpha=2$, so let us assume that $\alpha\geq 3$. Recall that $X$ is
an almost simple group with socle $S$ such that $X/S$ is a
$2$-group. Viewing $\chi$ as a character of $G:=SL_n^\epsilon(q)$
and using Lemma \ref{lem:abc}, we have that $\chi$ extends to some
$\widetilde{\chi}\in\Irr_{2'}(\widetilde{G})$, where
$\widetilde{G}:=GL_n^\epsilon(q)$. By Proposition
\ref{prop:technical2}, we may and will choose $\widetilde{\chi}$ so
that
\[\lev(\widetilde{\chi})=\lev(\chi)=\alpha.\]
By Lemma \ref{lem:GKNT}, we have \[
\widetilde{\chi}=S(s_1,\lambda_1)\circ
S(s_2,\lambda_2)\circ\cdots\circ S(s_k,\lambda_k),\] where
$n=n_1+n_2+\cdots+n_k$ is a proper decomposition of $n$ and, for
$1\leq i\leq k$, $s_i\in C_{q-\epsilon}$  are pairwise different and
$\lambda_i$ is a partition of $n_i$.

We now consider
\[
\widetilde{\chi}_1:=S(s_1^2,(n_1))\circ
S(s_2^2,(n_2))\circ\cdots\circ S(s^2_k,(n_k))
\]
and
\[\chi_1:=\widetilde{\chi}_1\hspace{-5pt}\downarrow_G.\] We claim that $\chi_1$ is
odd-degree, irreducible, and trivial on $\bZ(G)$, and thus can be
viewed as a member of $\Irr_{2'}(S)$; furthermore, $\chi_1$ has
$2$-rationality level $\alpha-1$ and is $P$-invariant.

\bigskip

(a) We note that the Jordan decomposition of $\widetilde{\chi}$ is
simply the pair $(s,\psi)$ where
\[s=\diag(\underbrace{s_1,...s_1}_{n_1 \text{ times}},\underbrace{s_2,...,s_2}_{n_2 \text{ times}},...,
\underbrace{s_k,...,s_k}_{n_k \text{ times}})\] and
\[\psi=\varphi^{\lambda_1}\otimes
\varphi^{\lambda_2}\otimes\cdots\otimes \varphi^{\lambda_k},\] a
unipotent character of the centralizer
\[
\bC_{GL_n^\epsilon(q)}(s)=GL_{n_1}^\epsilon(q)\times
GL_{n_2}^\epsilon(q)\times \cdots\times GL_{n_k}^\epsilon(q).
\]
Since $\widetilde{\chi}(1)$ is odd, the semisimple element $s$ is
$2$-central. It follows that $s^2$ is $2$-central, and so
$\widetilde{\chi}_1(1)$ is odd, implying that $\chi_1$ is
irreducible (by Lemma \ref{lem:abc}) and odd-degree.

Next, we note that $\chi$, being the restriction of
$\widetilde{\chi}$ to $G$, must belong to the Lusztig series of $G$
associated to $s^\ast\in G^\ast:=PGL_n^\epsilon(q)$, which is the
image of $s$ under the natural projection from $\widetilde{G}$ onto
$G^\ast$. Since $\chi$ is trivial on $\bZ(G)$, Lemma
\ref{lem:restriction-center} implies that $s^\ast\in
[G^\ast,G^\ast]=S$. Therefore $(s^2)^\ast$, the image of $s^2$ under
the same projection, also belongs to $[G^\ast,G^\ast]$. It follows
that, $\chi_1$, a member of the series $\EC(G,(s^2)^\ast)$, is
trivial on $\bZ(G)$, again by Lemma \ref{lem:restriction-center}.

\bigskip

(b) Next we show that $\lev(\chi_1)=\alpha-1$. We identify each
$t=\tau I_n\in\bZ(\widetilde{G})$ with $\tau\in C_{q-\epsilon}$.
Then \[\ord(ts)=\lcm(\ord(\tau s_1),\ord(\tau s_2),...,\ord(\tau
s_k)),\] and so
\[\ord(ts)_2=\max_{1\leq i\leq k}\{\ord(\tau s_i)_2\}.\]
Therefore, by Proposition \ref{prop:technical2},
\[
\alpha=\lev(\chi)=\log_2\min_{\tau\in C_{q-\epsilon}}\{\max_{1\leq
i\leq k}\{\ord(\tau s_i)_2\}\}.
\]
Note that $\alpha=\lev(\widetilde{\chi})=\log_2(\max_{1\leq i\leq
k}\{\ord(s_i)_2\})$. We deduce that
\[
\max_{1\leq i\leq k}\{\ord(s_i)_2\}\leq \max_{1\leq i\leq
k}\{\ord(\tau s_i)_2\}
\]
for every $\tau\in C_{q-\epsilon}$. This implies that
\[
\max_{1\leq i\leq k}\{\ord(s_i^2)_2\}\leq \max_{1\leq i\leq
k}\{\ord(\tau s_i^2)_2\}
\]
for every $\tau\in C_{q-\epsilon}$. (The inequality is certainly
true if $\tau$ is a square in $C_{q-\epsilon}$. Otherwise, we have
$\ord(\tau)_2=(q-\epsilon)_2$, and so $\ord(\tau
s_i^2)_2=(q-\epsilon)_2$ and the inequality is still satisfied.)
Using Proposition \ref{prop:technical2} again, we now have
\[
\lev(\widetilde{\chi}_1)\leq \lev(\chi_1).
\]
On the other hand, we have $\lev(\widetilde{\chi}_1)\geq
\lev(\chi_1)$ by Lemma \ref{lem:1}(iii) and
\[\lev(\widetilde{\chi}_1)=\log_2(\ord(s^2)_2)=\alpha-1\] by
Proposition \ref{prop:technical} and a note that $\alpha\geq 3$. We
conclude that $\lev(\chi_1)=\alpha-1$.

\bigskip

(c) Finally we argue that $\chi_1$ is $\varphi$-invariant for every
automorphism $\varphi\in X$. We will identify $\varphi$ as an
automorphism of $G$ via the natural identification $\Aut(S)\cong
\Aut(G)$ (see \cite[Corollary 5.1.4]{Gorensteinetal}). Recall that
$\chi$ is $\varphi$-invariant and, when viewed as a character of
$G$, belongs to $\EC(G,s^\ast)$, where $s^\ast$ is the image of $s$
under the projection $\widetilde{G}\rightarrow G^\ast$. We have that
$\EC(G,s^\ast)$ is $\varphi$-invariant, and therefore the
$G^\ast$-conjugacy class of $s^\ast$ is $\varphi^\ast$-invariant,
where $\varphi^\ast$ is the dual automorphism of $\varphi$ defined
just before (\ref{eq:2}). It follows that the $G^\ast$-conjugacy
class of $(s^2)^\ast$ is also $\varphi^\ast$-invariant, which in
turn implies that the series $\EC(G,(s^2)^\ast)$ is
$\varphi$-invariant.

Recall that $\iota:SL_n(\overline{\FF_r})\hookrightarrow
GL_n(\overline{\FF_r})$, where $r$ is the (only) prime divisor of
$q$, is a regular embedding, in the sense of \cite[\S 1.7]{GM20}.
Also, $\chi_1$ is the unique irreducible constituent of the
restriction of the semisimple character
$\widetilde{\chi}_1\in\Irr(\widetilde{G})$ to $G$. It follows from
\cite[Corollary 2.6.18]{GM20} that $\chi_1$ is the unique semisimple
character in $\EC(G,(s^2)^\ast)$, and hence must be
$\varphi$-invariant.

We have produced the character $\chi_1\in\Irr_{2'}(S)$ of level
$\alpha-1$ that satisfies all the required conditions, provided that
$\alpha\geq 3$. Repeating the process, one can show that $S$
possesses similar characters of every level from $2$ to $\alpha$.
\end{proof}


\section{Groups of type $E_6$ and ${}^2E_6$}\label{sec:E6}

In this section we prove Theorem \ref{thm:continuity2} for
$S=E_6^\epsilon(q)$ with $q$ odd. We first establish a version of
Proposition \ref{prop:technical} for the groups in question. The
version for all $p$ generally is yet to be determined, but,
fortunately, still holds true when $p=2$.

\begin{proposition}\label{prop:technicalE6}
Let $\alpha\in\ZZ^{\geq 2}$. Let $G=E_6^\epsilon(q)_{sc}$ with $q$
odd and ${\chi}\in\Irr_{2'}(G)$. Suppose that ${\chi}\in\EC({G},s)$
for some semisimple element $s\in G^\ast:=E_6^\epsilon(q)_{ad}$.
Then $\lev({\chi})=\alpha$ if and only if $\ord(s_2)= 2^\alpha$.
\end{proposition}

\begin{proof} We reuse some notation in Section \ref{sec:simple-groups}. In particular, we view $G$ as
the group of fixed points of a simple simply connected algebraic
group $\bG$ (of type $E_6$ and in odd characteristic) under a
Frobenius endomorphism $F$ on $\bG$. Let $(\bG^\ast,F^\ast)$ be a
dual pair of $(\bG,F)$, so that $G^\ast={\bG^\ast}^{F^\ast}$.

We first note that odd-degree unipotent characters of $G$ lie in the
principal series and are $2$-rational, as mentioned in the proof of
\cite[Theorem 3.4]{Malle19}. Therefore we may assume that $s$ is
nontrivial. As seen in Proposition \ref{prop:technical-0}, the
semisimple element $s$ is $2$-central in $G^\ast$. It is argued in
the proof of \cite[Theorem 5.1]{Navarro-Tiep21} that
$\bC_{\bG^\ast}(s)$ is the same for all such elements $s$ and equal
to an $F^\ast$-stable Levi subgroup of $\bG^\ast$, say $\bL^\ast$,
of type $D_5T_1$. (There are precisely $q-\epsilon-1$ conjugacy
classes of such elements $s$ in $G^\ast$.) The centralizer
$\bC_{G^\ast}(s)$ is therefore isomorphic to
$D^\epsilon_5(q).C_{q-\epsilon}$, whose unipotent characters are of
the form $\psi\circ f$, where $\psi$ is a unipotent character of
$D^\epsilon_5(q)$ and $f$ is the natural morphism from $\bL^\ast$ to
its $D_5$-factor, by \cite[Proposition 13.20]{Digne-Michel91}. By
\cite[Corollary 1.12]{Lusztig02}, unipotent characters of
$D^\epsilon_5(q)$ are rational-valued (there are indeed exactly
eight of them of odd-degree), and so the same thing holds for
$\bC_{G^\ast}(s)$. Now Theorem \ref{thm:technical} is applied to
yield the result.
\end{proof}

\begin{proposition}\label{prop:E6}
Conjecture \ref{conj:continuity2-repeated} holds true for $p=2$ and
$S=E_6^\epsilon(q)$ with $q$ odd.
\end{proposition}

\begin{proof}
We keep the notation in Proposition \ref{prop:technicalE6}. Suppose
that $\chi\in\Irr_{2'}(S)$ with $\lev(\chi)=\alpha$ and $\chi$ is
$X$-invariant. Here $X/S$ is a $2$-group of outer automorphisms of
$S$. Viewing $\chi$ as a character of $G$, we have $\chi\in\EC(G,s)$
for some $2$-central semisimple element $s\in G^\ast$ with
$\ord(s_2)=2^\alpha$. Also, by Lemma \ref{lem:restriction-center},
$s\in [G^\ast,G^\ast]=S$. Furthermore, for every $\varphi\in X$, $s$
and $\varphi^\ast(s)$ are conjugate in $G^\ast$, by formula
(\ref{eq:2}).

For each $0\leq i\leq \alpha-2$, consider the ($2$-central
semisimple) elements
\[s_i:=s^{2^i}\in G^\ast.\] By \cite[Table 1]{Malle19}, the
centralizer $\bC_{G^\ast}(s_i)$ is connected, and thus the Lusztig
series $\EC(G,s_i)$ contains a unique semisimple character, which we
denote by $\chi_{(s_i)}$ (see \cite[Definition 14.39 and Corollary
14.47]{Digne-Michel91}). Now $\chi_{(s_i)}$ is odd-degree (since
$s_i$ is $2$-central), restricts trivially on $\bZ(G)$ (by Lemma
\ref{lem:restriction-center}), and $X$-invariant (by (\ref{eq:2})
and the uniqueness of semisimple character). Moreover,
\[\lev(\chi_{(s_i)})=\log_2(\ord(s_i)_2)=\alpha-i.\] The proof is
complete.
\end{proof}

Theorem \ref{thm:continuity2} is now completely proved, by
Propositions \ref{prop:linhtinh}, \ref{prop:SL}, and \ref{prop:E6}.


\section{A $p$-local lower bound for
$|\Irr_{p'}(G)|$}\label{sec:lower-bound}

In this section we discuss Conjecture \ref{conj:Irrp'-bound} and
prove it for $p=2$.

As already mentioned in the introduction, Conjecture
\ref{conj:Irrp'-bound} arises from three different pieces: the
continuity property Conjecture \ref{conj:continuity1},
Navarro-Tiep's conjecture on fields of values of $p'$-degree
characters, and Hung-Malle-Mar\'{o}ti's bound for the number of
almost $p$-rational irreducible characters. For reader's
convenience, we recall the latter two here.

\begin{conjecture}[\cite{Navarro-Tiep21}]\label{conj:Navarro-Tiep}
Let $\chi\in\Irr_{p'}(G)$ with $\lev_p(\chi)=\alpha\in\ZZ^{\geq 0}$.
Then $p$ does not divide $[\QQ_{p^\alpha}:(\QQ(\chi)\cap
\QQ_{p^\alpha})]$.
\end{conjecture}

We note that Conjecture \ref{conj:Navarro-Tiep} is trivial when
$\alpha\leq 1$. Also, when $\alpha\geq 1$, its conclusion implies
that $[\QQ(\chi):\QQ]$ is divisible by $p^{\alpha-1}$, which in turn
implies that $G$ has at least $p^{\alpha-1}$ different irreducible
characters of the same degree and field of values as $\chi$. This
bound is significant when $\alpha\geq 2$ but not so when $\alpha=1$.
The following provides a bound for characters of levels $0$ and $1$.

\begin{theorem}[\cite{Hung-Malle-Maroti}]   \label{theorem-p'-degree}
 Let $G$ be a finite group, $p$ a prime dividing the order of $G$ and~$P$ a
 Sylow $p$-subgroup of $G$. Then the number of almost
$p$-rational irreducible characters of $p'$-degree of $G$ is at
least $2\sqrt{p-1}$.
 Moreover, the following are equivalent:
 \begin{itemize}
  \item[(i)] the number of almost
$p$-rational irreducible characters of $p'$-degree of $G$ is
$2\sqrt{p-1}$;
  \item[(ii)] the number of almost
$p$-rational irreducible characters of $p'$-degree of $\bN_G(P)$ is
$2\sqrt{p-1}$;
  \item[(iii)] $P$ is cyclic and $\bN_G(P)$ is isomorphic to the Frobenius
   group $P\rtimes C_{\sqrt{p-1}}$.
 \end{itemize}
\end{theorem}

\begin{theorem}\label{thm:bound}
Conjecture \ref{conj:Irrp'-bound} follows from Conjectures
\ref{conj:continuity1} and \ref{conj:Navarro-Tiep}.
\end{theorem}

\begin{proof}
Assume that Conjectures \ref{conj:continuity1} and
\ref{conj:Navarro-Tiep} hold true for $G$ and $p$. We aim to
establish the inequality
\[|\Irr_{p'}(G)|\geq
\frac{\exp(P/P')-1}{p-1}+2\sqrt{p-1}-1,\] provided that $|G|$ is
divisible by $p$. As explained in Remark \ref{remark:1}, by Theorem
\ref{theorem-p'-degree}, this is reduced to showing
\[
|\{\chi\in\Irr_{p'}(G): \lev(\chi)\geq 2\}|\geq
\frac{\exp(P/P')-1}{p-1}-1.
\]
We may and will assume that $\exp(P/P')\geq p^2$. The wanted
inequality is further reduced to
\[
|\{\chi\in\Irr_{p'}(G): \lev(\chi)=\beta\}|\geq p^{\beta-1},
\]
for every $2\leq \beta\leq \log_p(\exp(P/P'))$.

By \cite[Theorem B]{Navarro-Tiep19}, there exists a $p'$-degree
irreducible character of $G$ of $p$-rationality level at least
$\log_p(\exp(P/P')\geq 2$. The continuity property from Conjecture
\ref{conj:continuity1} then implies that $G$ possesses $p'$-degree
irreducible characters of all levels from $2$ to
$\log_p(\exp(P/P'))$.

Let $\chi\in\Irr_{p'}(G)$ with $\lev(\chi)=\beta$. By Conjecture
\ref{conj:Navarro-Tiep}, $p$ does not divide
$[\QQ_{p^\beta}:(\QQ(\chi)\cap \QQ_{p^\beta})]$, implying that
$[\QQ(\chi):\QQ]$ is divisible by $p^{\beta-1}$. The number of
Galois conjugates of $\chi$ is $[\QQ(\chi):\QQ]$, by \cite[Theorem
3.1]{Navarro18}, and therefore is at least $p^{\beta-1}$. Also, each
Galois conjugate of $\chi$ has the same degree and field of values
as $\chi$. We deduce that the number of $p'$-degree irreducible
characters of $p$-rationality level $\beta$ is at least
$p^{\beta-1}$, as desired.
\end{proof}

\begin{theorem}\label{thm:Irr2'-bound-repeated}
Let $G$ be finite group and $P\in\Syl_2(G)$. Then
\[|\Irr_{2'}(G)|\geq \exp(P/P').\] Moreover, the equality occurs if
and only if $P$ is cyclic and self-normalizing.
\end{theorem}

\begin{proof}
Conjecture \ref{conj:Navarro-Tiep} has been verified for $p=2$ in
\cite[Theorem A1]{Navarro-Tiep21}. Therefore, the inequality follows
from Theorem \ref{thm:continuity} and Theorem \ref{thm:bound}.

The if direction of the second statement of the theorem is clear, by
using the cyclic-Sylow case of the McKay conjecture \cite{Dade96}.
So assume that $|\Irr_{2'}(G)|= \exp(P/P')$. Following the proof of
Theorem \ref{thm:bound}, we see that $G$ then has precisely two
almost $2$-rational odd-degree irreducible characters. By Theorem
\ref{theorem-p'-degree}, this occurs if and only if $P$ is cyclic
and $\bN_G(P)=P$, as stated.
\end{proof}

We conclude by a remark that, as the McKay-Navarro conjecture admits
a block-wise version (see the discussion after Conjecture 9.13 in
\cite{Navarro18}), Conjectures \ref{conj:continuity1} and
\ref{conj:Irrp'-bound} are expected to hold for the principal block
(and perhaps for all blocks of maximal defect as well). That is, if
the principal $p$-block $B$ of a finite group $G$ contains a
height-zero irreducible character of $p$-rationality level
$\alpha\geq 2$, then $B$ contains similar characters of all levels
from $2$ to $\alpha$. Moreover, the number of height-zero characters
in $B$ is at least $\frac{\exp(P/P')-1}{p-1}+2\sqrt{p-1}-1$. We have
decided not to pursue these block-wise versions here.

\end{document}